\documentclass[12pt]{amsart}

\usepackage{enumitem}
\usepackage{psfrag}
\usepackage{graphicx}
\usepackage{pinlabel}
\usepackage{graphicx}
\usepackage{amsmath}
\usepackage{amssymb}
\usepackage{amscd}
\usepackage{autobreak}
\usepackage{picinpar}
\usepackage{times}
\usepackage{pb-diagram}
\usepackage{graphicx}
\usepackage{wrapfig}
\usepackage{xspace}
\usepackage{hyperref}
\usepackage{url}
\usepackage{caption}
\usepackage{subcaption}
\usepackage{fancyhdr}
\usepackage{overpic}
\usepackage{color}
\usepackage[square,comma,sort&compress,numbers]{natbib}
\usepackage{latexsym}
\usepackage{mathrsfs}
\usepackage{amsfonts}
\usepackage{hyperref}
\usepackage{amsthm}
\usepackage{appendix}
\usepackage{pdfsync}
\usepackage{float}

\theoremstyle{definition}
\newtheorem{theorem}{Theorem}[section]
\newtheorem{lemma}[theorem]{Lemma}
\newtheorem{proposition}[theorem]{Proposition}
\newtheorem{corollary}[theorem]{Corollary}
\newtheorem{definition}[theorem]{Definition}

\newtheorem{remark}[theorem]{Remark}

\newtheorem*{theorem*}{Theorem}
\graphicspath{{figures/}}   

\begin{document}

\title{\bf Spherical orthogonal ring patterns on surfaces and modified combinatorial total geodesic curvatures}
\author{Zhiwen Xiong, Xu Xu}
	
\date{\today}
	
\address{School of Mathematics and Statistics, Wuhan University, Wuhan, 430072, P.R.China}
\email{xiongzhiwen@whu.edu.cn}

\address{School of Mathematics and Statistics, Wuhan University, Wuhan, 430072, P.R.China}
\email{xuxu2@whu.edu.cn}
	
	
	
\begin{abstract}
Orthogonal ring patterns are natural generalizations of circle patterns.
Bobenko-Hoffmann-R\"orig \cite{B-H-R} and Bobenko \cite{B} established the variational principles of the classical combinatorial curvature for the Euclidean, hyperbolic and spherical orthogonal ring patterns.
Bobenko-Hoffmann-R\"orig's work \cite{B-H-R} and Bobenko's work \cite{B} imply the rigidity of Euclidean and hyperbolic orthogonal ring patterns on closed surfaces, while the rigidity of spherical orthogonal ring patterns on closed surfaces is not known.
In this paper, we study the spherical orthogonal ring patterns on closed surfaces with cellular decompositions satisfying certain necessary conditions.  Using a modification of the combinatorial total geodesic curvature introduced by Nie in \cite{Nie}, we prove the rigidity of spherical orthogonal ring patterns on closed surfaces by variational principles.
\end{abstract}
	
\maketitle
		
\section{Introduction}

\subsection{Spherical orthogonal ring patterns}\label{Spherical ring patterns}
Spherical orthogonal ring patterns were first introduced by Tellier, Hauswirth, Douthe and Baverel \cite{T-H-D-B} for the study of doubly-curved building envelopes.
Motivated by \cite{T-H-D-B}, Bobenko, Hoffmann and R\"orig \cite{B-H-R} introduced the concept of orthogonal ring pattern in the Euclidean plane as natural generalizations of circle patterns,  and established its variational principles with respect to the classical combinatorial curvature.
In \cite{B}, Bobenko generalized the Euclidean orthogonal ring pattern to the sphere and the hyperbolic plane, and established the corresponding variational principles with respect to the classical combinatorial curvature.
In \cite{B}, \cite{B-H-R} and \cite{T-H-D-B}, the cellular decompositions corresponding to the orthogonal ring patterns are subsets of the $\mathbb{Z}_2$ lattice. In this paper, we consider the spherical case and extend the framework of orthogonal ring patterns to cellular decompositions on closed surfaces satisfying certain necessary conditions.

Let $S$ be a closed surface with a cellular decomposition $\Sigma$, in which every face has an even number of edges and every vertex
is adjacent to an even number of edges. We also say that the vertex has an even degree in this case. Let $V$, $E$, $F$ denote the sets of vertices, edges and faces of $\Sigma$, respectively. For each face $f_i\in F$, select a point $O_i$ in $f_i$. Given a radius function $r: F\to (0,\frac{\pi}{2}) $ that maps each face $f_i$ to a value $r_i$, each face $f_i$ corresponds uniquely to a ring centered at $O_i$. The outer radius $R_i$ and the inner radius $r_i$ of this ring satisfy the equation
\begin{equation}\label{equRr}
\cos R_i=q\cos r_i,
\end{equation}
where $q$ is a constant satisfying $0<q\le 1$.

For each edge $e\in E$, suppose $e$ is incident to both faces $f_i$ and $f_j$. We require that the inner circle of $f_i$ is orthogonal to the outer circle of $f_j$, and concurrently, the outer circle of $f_i$ is orthogonal to the inner circle of $f_j$. Please refer to Figure \ref{Figure 1}.
This configuration gives rise to a spherical quadrilateral $O_iv_1O_jv_2$, which satisfies the following two conditions:
\begin{enumerate}
	\item[(C1)] $\angle O_iv_1O_j=\angle O_iv_2O_j=\frac{\pi}{2}$;
	\item[(C2)] $|O_iv_1|=r_i$, $|O_iv_2|=R_i$, $|O_jv_1|=R_j$, $|O_jv_2|=r_j$.
\end{enumerate}
Here $v_1$ and $v_2$ are two intersection points arising from these orthogonality conditions. Note that the equation (\ref{equRr}) ensures the existence of the spherical quadrilateral $O_iv_1O_jv_2$. Therefore, each edge $e\in E$ corresponds uniquely to such a spherical quadrilateral. By gluing all such spherical quadrilaterals isometrically along their corresponding edges, we obtain a spherical orthogonal ring pattern on the surface $S$.

\begin{figure}[htbp]
	\centering
	\includegraphics[scale=1]{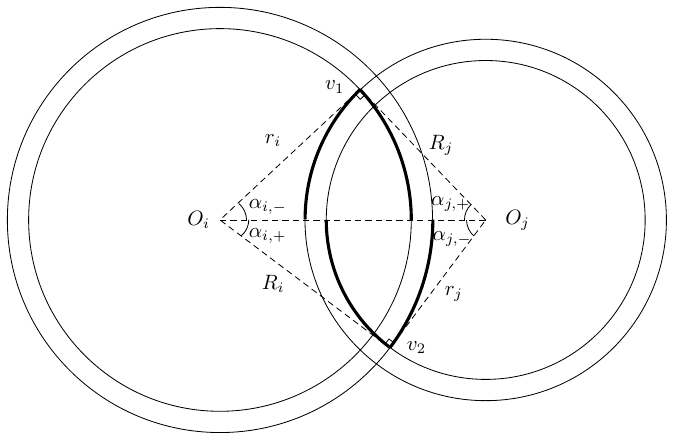}
	\caption{The spherical quadrilateral $O_iv_1O_jv_2$ associated with edge $e$}
	\label{Figure 1}
\end{figure}

\begin{definition}
Let $S$ be a closed surface with a cellular decomposition $\Sigma=(V,E,F)$, such that every face in $\Sigma$ has an even number of edges and every vertex has even degree. A spherical orthogonal ring pattern on the closed surface $(S,\Sigma)$ is a collection of finitely many rings $T_1$, $\cdots$, $T_{|F|}$ satisfying the following conditions:
\begin{enumerate}
	\item[(1)] Each ring corresponds bijectively to a face in $F$, and every ring $T_i$ is centered at $O_i$;
	\item[(2)] For every edge $e\in E$, assume $e$ is incident to both faces $f_i$ and $f_j$. Then the rings $T_i$ and $T_j$ are orthogonal, i.e., the outer circle of $T_i$ is orthogonal to the inner circle of $T_j$, and the inner circle of $T_i$ is orthogonal to the outer circle of $T_j$.
	\item[(3)] Conical singularities are permitted at both the center of each ring and the intersection points of any two rings.
\end{enumerate}
\end{definition}

\begin{remark}
	The requirement that every face in $\Sigma$ has an even number of edges arises from gluing the spherical quadrilaterals along the edges of each face $f_i$ with alternating side lengths $R_i$ and $r_i$, as illustrated in Figure \ref{Figure 3} (left). Similarly, the even degree of every vertex is necessary because, when spherical quadrilaterals are glued successively along edges incident to the vertex, the long side $R_i$ must alternate with the short side $r_j$, for any $i, j\in\{1,...,|F|\}$ such that the faces $f_i$ and $f_j$ share an edge. Figure \ref{Figure 3} (middle) demonstrates the issue that arises when a vertex has odd degree, while Figure \ref{Figure 3} (right) shows the valid local configuration around the vertex with even degree.
\end{remark}

\begin{figure}[htbp]
	\centering
	\includegraphics[scale=1.2]{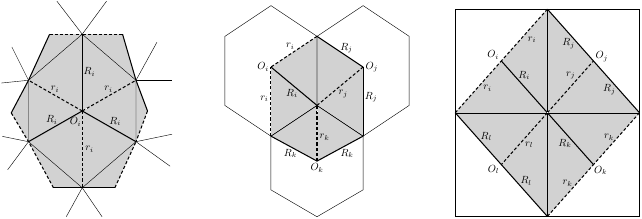}
	\caption{Local construction of spherical orthogonal ring patterns}
	\label{Figure 3}
\end{figure}

Bobenko, Hoffmann and R\"orig \cite{B-H-R} defined the combinatorial curvature of the Euclidean orthogonal ring patterns at each point $O_i$ as
\begin{equation}\label{curvature}
	\Theta_i=\sum_{j:f_j\sim f_i}\varphi_{ij},
\end{equation}
where $\varphi_{ij}$ denotes the angle at $O_i$ of the corresponding quadrilateral in the Euclidean plane and $f_j\sim f_i$ means that faces $f_j$ and $f_i$ share an edge.
Their work implies the rigidity of the Euclidean orthogonal ring pattern on surfaces.
In \cite{B}, Bobenko further defined the combinatorial curvature at each point $O_i$ in the hyperbolic case using the form of equation (\ref{curvature}),
and likewise his work
implies the rigidity of hyperbolic orthogonal ring patterns on surfaces.
However, if the curvature defined by equation (\ref{curvature}) is applied to the spherical setting, the corresponding rigidity results cannot be obtained. Bobenko \cite{B} explained that this is because the Hessian matrix of the potential function associated with this curvature is not positive definite, and thus the convexity is lacking. In this paper, we introduce a new combinatorial curvature, called the modified combinatorial total geodesic curvature, to establish the rigidity of spherical orthogonal ring patterns.

\subsection{Modified combinatorial total geodesic curvature for spherical orthogonal ring patterns.}
In the context of spherical geometry, the geodesic curvature of a circle with radius $r_i$ is $\cot r_i$. By connecting the centers $O_i$ and $O_j$ with a geodesic segment, the spherical quadrilateral $O_iv_1O_jv_2$ is divided into two triangles $\triangle O_iv_1O_j$ and $\triangle O_iv_2O_j$. We denote the two other angles of triangle $\triangle O_iv_1O_j$ as $\alpha_{i,-}$ and $\alpha_{j,+}$, and those of triangle $\triangle O_iv_2O_j$ as $\alpha_{i,+}$ and $\alpha_{j,-}$, as illustrated in Figure \ref{Figure 1}.

For the ring $T_i$, consider the integral of the geodesic curvature of its inner circle along the arc contained in triangle $\triangle O_iv_1O_j$ (shown as the bold arc on the inner circle of $T_i$ in Figure \ref{Figure 1}), and the integral of the geodesic curvature of its outer circle along the arc contained in triangle $\triangle O_iv_2O_j$ (shown as the bold arc on the outer circle of $T_i$ in Figure \ref{Figure 1}). We denote the sum of these two integrals as $L_{(e,i)}$. It is straightforward to verify that
\begin{equation*}\label{L_ij}
	L_{(e,i)}=\alpha_{i,-}\cos r_i+\alpha_{i,+}\cos R_i.
\end{equation*}
Similarly, for the ring $T_j$, denote the corresponding sum of integrals as $L_{(e,j)}$. We also have
\begin{equation*}\label{L_ji}
	L_{(e,j)}=\alpha_{j,+}\cos R_j+\alpha_{j,-}\cos r_j.
\end{equation*}
From this, we can define the modified combinatorial total geodesic curvature at each point $O_i$ as
\begin{equation*}\label{K_i}
	L_i=\sum_{e:e<f_i}L_{(e,i)},
\end{equation*}
where $i\in\{1,\cdots, |F|\}$ and the notation $e<f_i$ indicates that $e$ is an edge of the face $f_i$.

\begin{remark}
Here, our approach is analogous to the Euclidean and hyperbolic settings, where the curvature $\Theta_i$ and the variable $r_i$ are respectively replaced by their counterparts in our framework, which are the modified combinatorial total geodesic curvature $L_i$ and the variable $u_i$ (introduced in Section \ref{Configurations of orthogonal rings}). In particular, in the Euclidean setting, the geodesic curvature of a circle of radius $r_i$ is $\frac{1}{r_i}$. In this case, the modified combinatorial total geodesic curvature $L_i$ reduces to the sum of angles $\Theta_i$ as shown in equation (\ref{curvature}). Different from the Euclidean case, the modified combinatorial total geodesic curvature $L_i$ does not equal $\Theta_i$ in either the spherical or the hyperbolic setting.
\end{remark}

\begin{remark}
In \cite{Nie}, Nie first introduced the combinatorial total geodesic curvature and, using the variational methods, established the rigidity and existence of circle patterns in the spherical setting. Ba-Hu-Sun \cite{B-H-S} then generalized the combinatoiral total geodesic curvature to tangential hyperbolic circle packings. After that, Hu-Lu-Tan-Zhong-Zhou \cite{H-L-T-Z-Z} established the convergence of combinatorial Ricci flows to degenerated circle packings in hyperbolic background geometry via combinatorial total geodesic curvature. From equation (\ref{equRr}), it follows that when $q=1$, the outer radius $R_i$ and the inner radius $r_i$ of every ring in the spherical orthogonal ring pattern are equal. In this case, the modified combinatorial total geodesic curvature we defined above coincides with the combinatorial total geodesic curvature introduced by Nie in \cite{Nie}. In particular, the spherical orthogonal ring pattern also becomes equivalent to the circle pattern constructed by Bobenko and Springborn in \cite{B-S}.
\end{remark}

\subsection{The main theorem.}

In \cite{B-S}, Bobenko and Springborn introduced a new construction of circle patterns and proved the rigidity and existence in both Euclidean and hyperbolic settings via the variational principles. The rigidity and existence for the spherical case were later established by Nie \cite{Nie}
using the combinatorial total geodesic curvatures. Guo and Luo \cite{G-L} generalized Bobenko and Springborn's construction by replacing the circle centers and intersection points from hyperbolic points to ideal and hyperideal points, and proved the rigidity.
In \cite{X-X} and \cite{X-X-Z}, we considered generalized hyperbolic circle patterns on surfaces with boundary and cusps and established the existence of generalized circle patterns.

By replacing circles with rings, it follows from the discussion in Section \ref{Spherical ring patterns} that the curvature defined by equation (\ref{curvature}) fails to yield the rigidity for spherical orthogonal ring patterns.
In this paper, we introduce the modified combinatorial total geodesic curvature, and apply it to establish the rigidity of spherical orthogonal ring patterns on closed surfaces by variational principles.

\begin{theorem}\label{Thm of rigidity}
Let $S$ be a closed surface with a cellular decomposition $\Sigma$ in which every face has an even number of edges and every vertex has even degree. A spherical orthogonal ring pattern on the surface $S$ is uniquely determined by its modified combinatorial total geodesic curvature.
\end{theorem}

\begin{remark}
In \cite{B}, Bobenko \cite{B} established the variational principles for spherical orthogonal
ring patterns.
However, the rigidity was not obtained therein, since the Hessian matrix of the potential function associated with the traditional curvature introduced in \cite{B} is not positive definite or negative definite.
In this paper, we prove the rigidity of spherical orthogonal
ring patterns with respect to the modified combinatorial total geodesic curvature, thereby filling this gap.
In the case of traditional notion of discrete curvature, Izmestiev, Prosanov and Wu \cite{I-P-W} proved the existence and
rigidity of vertex scalings of convex spherical cone-metrics with prescribed positive curvature on the sphere. This raises two natural further questions. First, can the traditional notion of discrete curvature be used to prove the existence and rigidity of spherical circle patterns? Second, can one prove the existence and rigidity of spherical ring patterns using the traditional notion of discrete curvature?
\end{remark}

\subsection{Organization of the paper.} In Section \ref{Configurations of orthogonal rings}, we derive several key properties of orthogonal rings and establish a number of lemmas. In Section \ref{Rigidity of spherical orthogonal ring patterns}, we construct a convex function and provide a proof of Theorem \ref{Thm of rigidity}. In Section \ref{Combinatorial curvature flows of spherical orthogonal ring patterns}, we introduce the combinatorial Ricci flow and the combinatorial Calabi flow for spherical orthogonal ring patterns and establish the local convergence of these combinatorial curvature flows.
\\
\\
\textbf{Acknowledgements}\\[8pt]
The authors are supported by the National Natural Science Foundation of China under Grant No. 12471057.

\section{Configurations of orthogonal rings}\label{Configurations of orthogonal rings}

Given $r_i,r_j\in (0,\frac{\pi}{2})$ and $q\in (0,1]$, it follows from the spherical law of cosines that there exists a unique spherical quadrilateral $O_iv_1O_jv_2$ satisfying conditions (C1) and (C2) given in Section \ref{Spherical ring patterns}.

By applying the variable substitutions
\begin{equation}\label{u_i}
	u_i=\int_{r_0}^{r_i}\frac{\mathrm{d}t}{2q\cos t\sqrt{1-q^2\cos^2t}}
\end{equation}
for each $i\in \{1,\cdots, |F|\}$, we conclude that the admissible space $U$ for the spherical orthogonal ring pattern on the closed surface is a convex set. Here, $r_0$ is a constant in the interval $(0,\frac{\pi}{2})$. Indeed, we have
\begin{equation*}
	\frac{\mathrm{d} u_i}{\mathrm{d} r_i}=\frac{1}{\sin 2R_i},
\end{equation*}
which implies that $u_i$ increases with $r_i$. Since $r_i\in (0,\frac{\pi}{2})$,  it follows that $U$ is convex.

\begin{lemma}\label{1-form}
For the spherical quadrilateral $O_iv_1O_jv_2$, the 1-form $\omega=L_{(e,i)}\mathrm{d}u_i+L_{(e,j)}\mathrm{d}u_j$ is closed. In particular,
\[ \frac{\partial L_{(e,i)}}{\partial u_j}=\frac{\partial L_{(e,j)}}{\partial u_i}=2q\cos l_{ij}\cdot\frac{\sin r_i\sin r_j+\sin R_i\sin R_j}{\sin^2l_{ij}}>0. \]
\end{lemma}
\begin{proof}
From equation (\ref{equRr}), we have
$\cos R_i=q\cos r_i$ and $\cos R_j=q\cos r_j$, where $0<q\le 1$. Differentiating these relations yields
\begin{equation*}
\dfrac{\mathrm{d} R_i}{\mathrm{d} r_i}=q\dfrac{\sin r_i}{\sin R_i},\quad \dfrac{\mathrm{d} R_j}{\mathrm{d} r_j}=q\dfrac{\sin r_j}{\sin R_j}.
\end{equation*}

Considering the spherical triangle $\bigtriangleup O_iv_1O_j$, the spherical laws of sines and cosines in \cite{J-G} yield the following relations
\begin{equation}\label{sin law}
	\frac{\sin r_i}{\sin \alpha_{j,+}}=\frac{\sin R_j}{\sin \alpha_{i,-}}=\frac{\sin l_{ij}}{\sin \frac{\pi}{2}},
\end{equation}
\begin{equation}\label{cos law1}
	\cos \alpha_{i,-} = \frac{\cos R_j- \cos r_i \cos l_{ij}}{\sin r_i \sin l_{ij}},
\end{equation}
\begin{equation}\label{cos law2}
	\cos l_{ij}=\cos r_i\cos R_j+\sin r_i\sin R_j\cos\frac{\pi}{2},
\end{equation}
where $l_{ij}$ denotes the length of the geodesic segment $O_iO_j$.

Substituting the equation (\ref{cos law2}) and the second relation of equation (\ref{sin law}) into the equation (\ref{cos law1}) yields
\begin{equation}\label{alphai-}
	\cot\alpha_{i,-}=\cot R_j\sin r_i.
\end{equation}
Similarly, we obtain
\begin{equation}\label{alphaj+}
	\cot\alpha_{j,+}=\cot r_i\sin R_j.
\end{equation}
Taking the partial derivative of the equation (\ref{alphai-}) with respect to $r_j$ and the partial derivative of the equation (\ref{alphaj+}) with respect to $r_i$ yields
\begin{equation}\label{equ1}
	-\frac{1}{\sin^2\alpha_{i,-}}\frac{\partial \alpha_{i,-}}{\partial r_j}=-\frac{\sin r_i}{\sin^2 R_j}\cdot q\dfrac{\sin r_j}{\sin R_j} \Rightarrow \frac{\partial \alpha_{i,-}}{\partial r_j}=\frac{q\sin r_i\sin r_j}{\sin R_j\sin^2l_{ij}}=\frac{q\sin r_j\sin\alpha_{j,+}}{\sin R_j\sin l_{ij}} ,
\end{equation}
\begin{equation}\label{equ2}
	-\frac{1}{\sin^2\alpha_{j,+}}\frac{\partial \alpha_{j,+}}{\partial r_i}=-\frac{\sin R_j}{\sin^2r_i} \Rightarrow \frac{\partial \alpha_{j,+}}{\partial r_i}=\frac{\sin R_j}{\sin^2l_{ij}}=\frac{\sin\alpha_{i,-}}{\sin l_{ij}} .
\end{equation}

Now, consider the spherical triangle $\bigtriangleup O_iv_2O_j$. By similar arguments, we obtain
$$\cot\alpha_{i,+}=\cot r_j\sin R_i, \ \cot\alpha_{j,-}=\cot R_i\sin r_j.$$
Furthermore, by direct computations, we have
\begin{equation*}
	\frac{\partial \alpha_{i,+}}{\partial r_j}=\frac{\sin R_i}{\sin^2 l_{ij}}=\frac{\sin\alpha_{j,-}}{\sin l_{ij}} ,\quad
	\frac{\partial \alpha_{j,-}}{\partial r_i}=\frac{q\sin r_i\sin r_j}{\sin R_i\sin^2l_{ij}}=\frac{q\sin r_i\sin\alpha_{i,+}}{\sin R_i\sin l_{ij}} .
\end{equation*}

Computing the partial derivatives $\frac{\partial L_{(e,i)}}{\partial u_j}$ and $\frac{\partial L_{(e,j)}}{\partial u_i}$, we obtain
\begin{align*}
	\frac{\partial L_{(e,i)}}{\partial u_j}&= \sin 2R_j\frac{\partial L_{(e,i)}}{\partial r_j}\\
	&=\sin 2R_j(\frac{\partial \alpha_{i,-}}{\partial r_j}\cos r_i+\frac{\partial \alpha_{i,+}}{\partial r_j}\cos R_i) \\
	&=\sin 2R_j(\frac{q\sin r_i\sin r_j}{\sin R_j\sin^2 l_{ij}}\cdot\cos r_i+\frac{\sin R_i}{\sin^2 l_{ij}}\cdot\cos R_i)\\
	&=2\sin R_j\cos R_j\cdot\frac{\cos R_i}{\sin R_j}\cdot\frac{\sin r_i\sin r_j+\sin R_i\sin R_j}{\sin^2l_{ij}}\\
	&=2q\cos l_{ij}\cdot\frac{\sin r_i\sin r_j+\sin R_i\sin R_j}{\sin^2l_{ij}},
\end{align*}

\begin{align*}
	\frac{\partial L_{(e,j)}}{\partial u_i}&= \sin 2R_i\frac{\partial L_{(e,j)}}{\partial r_i}\\
	&=\sin 2R_i(\frac{\partial \alpha_{j,+}}{\partial r_i}\cos R_j+\frac{\partial \alpha_{j,-}}{\partial r_i}\cos r_j) \\
	&=\sin 2R_i(\frac{\sin R_j}{\sin^2 l_{ij}}\cdot\cos R_j+\frac{q\sin r_i\sin r_j}{\sin R_i\sin^2 l_{ij}}\cdot\cos r_j)\\
	&=2\sin R_i\cos R_i\cdot\frac{\cos R_j}{\sin R_i}\cdot\frac{\sin r_i\sin r_j+\sin R_i\sin R_j}{\sin^2l_{ij}}\\
	&=2q\cos l_{ij}\cdot\frac{\sin r_i\sin r_j+\sin R_i\sin R_j}{\sin^2l_{ij}}\\
	&=\frac{\partial L_{(e,i)}}{\partial u_j}>0.
\end{align*}
Thus, $\omega=L_{(e,i)}\mathrm{d}u_i+L_{(e,j)}\mathrm{d}u_j$ is a closed form and $\frac{\partial L_{(e,i)}}{\partial u_j}=\frac{\partial L_{(e,j)}}{\partial u_i}>0$.
\end{proof}

\begin{figure}[htbp]
	\centering
	\includegraphics[scale=1]{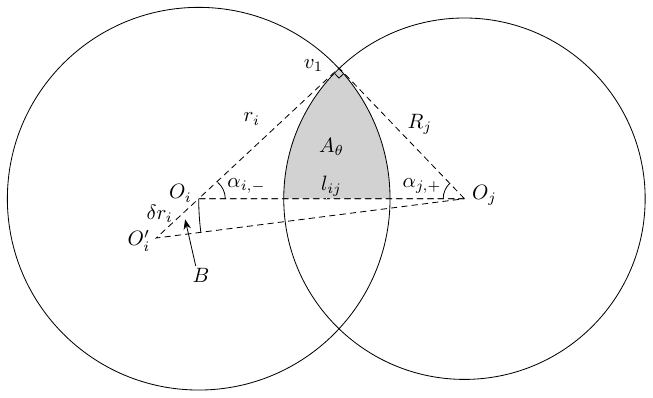}
	\caption{Illustration of $A_{\theta}$ and the variation of  $\triangle O_iv_1O_j$ under a small change}
	\label{Figure 2}
\end{figure}

In \cite{G-T}, Glickenstein and Thomas proved the variational formulas of the area and angles of triangles generated by discrete conformal structures on surfaces. Glickenstein \cite{Glickenstein} also employed the variational formulas to prove rigidity theorems for certain analogs of constant curvature and Einstein manifolds in the piecewise flat setting.  Here we only state the formula for spherical conformal structure as follows.

\begin{proposition}[\cite{G-T}, Proposition 9]\label{G-T formula}
	Given a spherical conformal structure, for any spherical triangle $\{i,j,k\}$ generated by the spherical conformal structure, we have
	\begin{equation*}
		\frac{\partial A_{ijk}}{\partial f_k}=\frac{\partial \gamma_i}{\partial f_k}(1-\cos l_{ik})+\frac{\partial \gamma_j}{\partial f_k}(1-\cos l_{jk}).
	\end{equation*}
	Here $A_{ijk}$ is the area of the triangle, $f=(f_i,f_j,f_k)$ is a function defined on  the vertices of the triangle, and $\gamma_i$, $\gamma_j$, $\gamma_k$ denote the inner angles at vertices $i$, $j$ and $k$ of the spherical triangle, respectively.
\end{proposition}

Following the approach of Proposition \ref{G-T formula} with appropriate modifications, we establish the following corollary.

\begin{corollary}\label{area}
	If \(A\) denotes the area of the triangle $\triangle O_iv_1O_j$, then we have
	\[ \frac{\partial A}{\partial r_i}=\frac{\partial\alpha_{j,+}}{\partial r_i}(1-\cos l_{ij}),   \quad \frac{\partial A}{\partial R_j}=\frac{\partial\alpha_{i,-}}{\partial R_j}(1-\cos l_{ij}). \]
\end{corollary}
\begin{proof}
	Since the area $A$ of the triangle $\triangle O_iv_1O_j$ is determined by $r_i$ and $R_j$, we first consider a variation $\delta r_i$ in the length of $r_i$ while keeping $R_j$ fixed. We only consider the case $\delta r_i>0$, as the case for $\delta r_i<0$ is analogous. This corresponds to moving the point $O_i$ along the direction of $v_1O_i$ with $v_1$ and $O_j$ fixed, as shown in Figure \ref{Figure 2}.
	The change in area $A$ is equal to the area of a circular sector plus  the area of another small region, denoted as $B$.
	
	Figure \ref{Figure 2} shows that $B$ is smaller than the area of a spherical sector of radius $\delta r_i$ and central angle $\alpha_{i,-}+\delta \alpha_{i,-}$, hence we have $B \le (\alpha_{i,-}+\delta \alpha_{i,-})(1-\cos \delta r_i)$.
	As $\delta r_i$ tends to $0^+$, we obtain
	\begin{align*}
		\lim_{\delta r_i\to 0^+}\frac{B}{\delta r_i}\le \lim_{\delta r_i\to 0^+}\frac{(\alpha_{i,-}+\delta \alpha_{i,-})(1-\cos \delta r_i)}{\delta r_i}.
	\end{align*}
	
	Since $\alpha_{i,-}+\delta \alpha_{i,-}$ is finite and $\lim_{\delta r_i\to 0^+}\frac{1-\cos\delta r_i}{\delta r_i}=0$, the right-hand side of the above inequality equals $0$, from which it follows that $\lim_{\delta r_i\to 0^+}\frac{B}{\delta r_i}=0$.
	Then we have
	\begin{align*}
		\lim_{\delta r_i\to 0^+}\frac{A(r_i+\delta r_i)-A(r_i)}{\delta r_i}
		&=\lim_{\delta r_i\to 0^+}\frac{\delta \alpha_{j,+}(1-\cos l_{ij})+B}{\delta r_i}\\
		&=\frac{\partial \alpha_{j,+}}{\partial r_i}(1-\cos l_{ij})+\lim_{\delta r_i\to 0^+}\frac{B}{\delta r_i}\\
		&=\frac{\partial \alpha_{j,+}}{\partial r_i}(1-\cos l_{ij}).
	\end{align*}	
	When $\delta r_i<0$, the left-hand limit also satisfies
$$\lim_{\delta r_i\to 0^-}\frac{A(r_i+\delta r_i)-A(r_i)}{\delta r_i}=\frac{\partial \alpha_{j,+}}{\partial r_i}(1-\cos l_{ij}).$$
	Combined with the result for the right-hand limit, this confirms that $$\frac{\partial A}{\partial r_i}=\lim_{\delta r_i\to 0}\frac{A(r_i+\delta r_i)-A(r_i)}{\delta r_i}=\frac{\partial \alpha_{j,+}}{\partial r_i}(1-\cos l_{ij}).$$
	Applying the same arguments to the point $O_j$, moving it along the direction of $v_1O_j$ with $v_1$ and $O_i$ fixed, we have
	\[ \frac{\partial A}{\partial R_j}=\frac{\partial\alpha_{i,-}}{\partial R_j}(1-\cos l_{ij}). \]
	This completes the proof of the corollary.
\end{proof}

By applying Corollary \ref{area} , we can obtain the following result.

\begin{lemma}\label{pd}
For the spherical triangle $\triangle O_iv_1O_j$, the following equations hold:
\begin{enumerate}
	\item[(1)] $\dfrac{\partial A_{\theta}}{\partial r_i}=\dfrac{1}{2}\sin r_i(2\alpha_{i,-}-\sin 2\alpha_{i,-})$,
	\item[(2)] $\dfrac{\partial A_{\theta}}{\partial r_j}
	=\dfrac{q}{2}\sin r_j(2\alpha_{j,+}-\sin 2\alpha_{j,+})$,
\end{enumerate}
where $A_{\theta}$ denotes a half of the area of the intersection region between the inner circle of ring $T_i$ and the outer circle of ring $T_j$, as illustrated by the shadowed region in Figure \ref{Figure 2}.
Specially, we have $\dfrac{\partial A_{\theta}}{\partial r_i}>0$, $\dfrac{\partial A_{\theta}}{\partial r_j}>0$.
\end{lemma}
\begin{proof}
Note that $A=\alpha_{i,-}+\alpha_{j,+}+\frac{\pi}{2}-\pi$. Combining this identity with Corollary \ref{area}, we obtain
\begin{equation*}
	\frac{\partial\alpha_{j,+}}{\partial r_i}(1-\cos l_{ij})=\frac{\partial\alpha_{i,-}}{\partial r_i}+\frac{\partial\alpha_{j,+}}{\partial r_i}.
\end{equation*}
Therefore,
\begin{equation*}
	\frac{\partial\alpha_{i,-}}{\partial r_i}=-\cos l_{ij}\frac{\partial\alpha_{j,+}}{\partial r_i}.
\end{equation*}
Similarly, we have $\frac{\partial\alpha_{j,+}}{\partial R_j}=-\cos l_{ij}\frac{\partial\alpha_{i,-}}{\partial R_j}$, which implies $\frac{\partial\alpha_{j,+}}{\partial r_j}=-\cos l_{ij}\frac{\partial\alpha_{i,-}}{\partial r_j}$.

Let $A_1$ be the area of the circular sector where the left circle in Figure  \ref{Figure 2} intersects the triangle $\triangle O_iv_1O_j$, and let $A_2$ be the area of the corresponding sector formed by the right circle. Then we have
\[ A_1=\alpha_{i,-}(1-\cos r_i), \quad A_2=\alpha_{j,+}(1-\cos R_j), \quad A_{\theta}=A_1+A_2-A. \]

Taking the partial derivative of $A_{\theta}$ with respect to $r_i$ gives
\begin{align}
	\frac{\partial A_{\theta}}{\partial r_i}&=\frac{\partial A_1}{\partial r_i}+\frac{\partial A_2}{\partial r_i}-\frac{\partial A}{\partial r_i}\nonumber\\
	&=\frac{\partial \alpha_{i,-}}{\partial r_i}(1-\cos r_i)+\alpha_{i,-}\sin r_i+\frac{\partial \alpha_{j,+}}{\partial r_i}(1-\cos R_j)-\frac{\partial\alpha_{j,+}}{\partial r_i}(1-\cos l_{ij})\nonumber\\
	&=-\cos l_{ij}\frac{\partial\alpha_{j,+}}{\partial r_i}(1-\cos r_i)+\alpha_{i,-}\sin r_i+\frac{\partial \alpha_{j,+}}{\partial r_i}(1-\cos R_j)-\frac{\partial\alpha_{j,+}}{\partial r_i}(1-\cos l_{ij})\nonumber\\
	&=\frac{\partial\alpha_{j,+}}{\partial r_i}(\cos l_{ij}\cos r_i-\cos R_j)+\alpha_{i,-}\sin r_i\nonumber\\
	&=-\sin l_{ij}\sin r_i\cos \alpha_{i,-}\frac{\partial\alpha_{j,+}}{\partial r_i}+\alpha_{i,-}\sin r_i.\label{pd r_i}
\end{align}
Similarly, we have
\[ \frac{\partial A_{\theta}}{\partial R_j}=-\sin l_{ij}\sin R_j\cos \alpha_{j,+}\frac{\partial\alpha_{i,-}}{\partial R_j}+\alpha_{j,+}\sin R_j, \]
which implies
\begin{equation}\label{pd r_j}
	\frac{\partial A_{\theta}}{\partial r_j}=-\sin l_{ij}\sin R_j\cos \alpha_{j,+}\frac{\partial\alpha_{i,-}}{\partial r_j}+q\alpha_{j,+}\sin r_j.
\end{equation}
Substituting equations (\ref{equ1}) and (\ref{equ2}) into equations (\ref{pd r_i}) and (\ref{pd r_j}) respectively yields

\begin{align*}
	\frac{\partial A_{\theta}}{\partial r_i}=-\sin\alpha_{i,-}\sin r_i\cos\alpha_{i,-}+\alpha_{i,-}\sin r_i
	=\frac{1}{2}\sin r_i(2\alpha_{i,-}-\sin 2\alpha_{i,-}),
\end{align*}

\begin{align*}
	\frac{\partial A_{\theta}}{\partial r_j}
	=-\sin l_{ij}\sin R_j\cos \alpha_{j,+}\cdot \frac{q\sin r_j\sin\alpha_{j,+}}{\sin R_j\sin l_{ij}}+q\alpha_{j,+}\sin r_j
	=\frac{q}{2}\sin r_j(2\alpha_{j,+}-\sin 2\alpha_{j,+}).
\end{align*}

\end{proof}

\begin{lemma}\label{negative definite}
For the spherical quadrilateral $O_iv_1O_jv_2$, the matrix $\frac{\partial(L_{(e,i)},L_{(e,j)})}{\partial (u_i,u_j)}$is negative definite.
\end{lemma}
\begin{proof}
Let $B_{\theta}$ be half of the area of the intersection region between the outer circle of ring $T_i$ and the inner circle of ring $T_j$.
Applying the result of Lemma \ref{pd} to the triangle $\triangle O_iv_2O_j$ yields the following equations

\begin{equation*}
	\frac{\partial B_{\theta}}{\partial r_i}
	=\frac{q}{2}\sin r_i(2\alpha_{i,+}-\sin 2\alpha_{i,+})>0,
\end{equation*}
\begin{equation*}
	\frac{\partial B_{\theta}}{\partial r_j}=\frac{1}{2}\sin r_j(2\alpha_{j,-}-\sin 2\alpha_{j,-})>0.
\end{equation*}

From the above equations and the proof of Lemma \ref{pd}, it follows that $\frac{\partial A_{\theta}}{\partial r_i}$, $\frac{\partial A_{\theta}}{\partial r_j}$, $\frac{\partial B_{\theta}}{\partial r_i}$, and $\frac{\partial B_{\theta}}{\partial r_j}$ are all positive.
Employing the Gauss-Bonnet theorem for the shadowed region of the triangle $\triangle O_iv_1O_j$ and for the corresponding region of the triangle $\triangle O_iv_2O_j$ leads to
\[
\begin{cases}
	A_{\theta}+\alpha_{i,-}\cos r_i+\alpha_{j,+}\cos R_j=\frac{\pi}{2},\\
	B_{\theta}+\alpha_{i,+}\cos R_i+\alpha_{j,-}\cos r_j=\frac{\pi}{2}.
\end{cases}
\]
Adding the two equations above yields
$A_{\theta}+B_{\theta}=\pi-L_{(e,i)}-L_{(e,j)}$.
Therefore,
\[
\frac{\partial (L_{(e,i)}+L_{(e,j)})}{\partial r_i}=-\frac{\partial (A_{\theta}+B_{\theta})}{\partial r_i}<0,
\]
\[
\frac{\partial (L_{(e,i)}+L_{(e,j)})}{\partial r_j}=-\frac{\partial (A_{\theta}+B_{\theta})}{\partial r_j}<0
\]
by Lemma \ref{pd}.

Combining this with Lemma \ref{1-form}, we have $\frac{\partial L_{(e,i)}}{\partial u_j}=\frac{\partial L_{(e,j)}}{\partial u_i}>0$, $ \frac{\partial (L_{(e,i)}+L_{(e,j)})}{\partial u_i}<0$ and $\frac{\partial (L_{(e,i)}+L_{(e,j)})}{\partial u_j}<0 $.
Therefore, the matrix $\frac{\partial(L_{(e,i)},L_{(e,j)})}{\partial (u_i,u_j)}$ is symmetric and strictly diagonally dominant with negative diagonal entries, hence negative definite.
\end{proof}

\section{Rigidity of spherical orthogonal ring patterns}\label{Rigidity of spherical orthogonal ring patterns}

\subsection{Construction of the convex function.}

In this section, we construct a convex function and prove its convexity, which plays a crucial role in the subsequent proof of the rigidity theorem for spherical orthogonal ring patterns on closed surfaces.

For each $e\in E$ and $e=f_i\cap f_j$, we can define a potential function
\[ \mathcal{E}_{e}(u_{i},u_{j}) =\int^{(u_{i},u_{j})}(-L_{(e,i)})\mathrm{d}u_{i}+(-L_{(e,j)})\mathrm{d}u_{j}  .\]
Let $U_e$ be the admissible space of the spherical quadrilateral corresponding to $e$, and we have $ U_{e}\subset \mathbb{R}^{2}$.
From Lemma $\ref{1-form}$ and Lemma \ref{negative definite} we know that $\mathcal{E}_{e}(u_{i},u_{j})$ is well defined on $U_{e}$.

\begin{lemma}
	$\mathcal{E}_{e}(u_{i},u_{j})$ is strictly convex on $U_{e}$.
\end{lemma}
\begin{proof}	
The Hessian matrix of $\mathcal{E}_{e}(u_{i},u_{j})$ is
\[ H=\begin{pmatrix}
	-\frac{\partial L_{(e,i)}}{\partial u_{i}}	&-\frac{\partial L_{(e,i)}}{\partial u_{j}} \\
	-\frac{\partial L_{(e,j)}}{\partial u_{i}}	&-\frac{\partial L_{(e,j)}}{\partial u_{j}}
\end{pmatrix} .\]
By Lemma \ref{negative definite}, $H$ is positive definite and then $\mathcal{E}_{e}(u_{i},u_{j})$ is strictly convex on $U_{e}$.
\end{proof}

Based on the discussion in Section \ref{Configurations of orthogonal rings}, we know that the admissible space $U$ is a convex set in $\mathbb{R}^{|F|}$. Hence, we have the following corollary.

\begin{corollary}\label{convex function 1}
The function $\mathcal{E}(u)=\sum_{e:e=f_i\cap f_j\in E}\mathcal{E}_{e}(u_{i}, u_{j})$ is strictly convex on $U$, where $ u=(u_{1},..., u_{|F|}).$
\end{corollary}

\begin{proof}
By simplifying $\mathcal{E}(u)$, we obtain
\begin{align*}
	\mathcal{E}(u)&=\sum_{e:e=f_i\cap f_j\in E}\mathcal{E}_{e}(u_{i}, u_{j})\\
	&=\sum_{e:e=f_i\cap f_j\in E}\int^{(u_{i},u_{j})}(-L_{(e,i)})\mathrm{d}u_{i}+(-L_{(e,j)})\mathrm{d}u_{j}\\
	&=\int^{u}\sum_{f_i\in F}(\sum_{e:e< f_i}(-L_{(e,i)}))du_{i}\\
	&=\int^{u}\sum_{f_i\in F}(-L_i)du_{i}.\\
\end{align*}

We compute the first-order and second-order partial derivatives of the function $\mathcal{E}(u)$, giving
\begin{align*}
    &\frac{\partial \mathcal{E}(u)}{\partial u_{i}} =-L_{i}, \quad \forall f_i\in F ;\\
	&\frac{\partial^2\mathcal{E}(u)}{\partial u^2_i} = -\frac{\partial L_{i}}{\partial u_{i}}=-\sum_{e: e<f_i} \frac{\partial L_{(e,i)}}{\partial u_{i}}>0, \quad \forall f_i\in F ;\\
	&\frac{\partial^2\mathcal{E}(u)}{\partial u_i\partial u_j} = -\frac{\partial L_{i}}{\partial u_{j}}= -\frac{\partial (\sum_{e:e<f_i }L_{(e,i)})}{\partial u_{j}}=
	\begin{cases}
		-\frac{\partial L_{(e,i)}}{\partial u_{j}}<0	, &\text{if there exists} \ e=f_i\cap f_j\in E,\\
		0,  &\text{otherwise}.
	\end{cases}
\end{align*}

Suppose $J$ is the Hessian matrix of $\mathcal{E}(u)$. Then we have
\begin{align*}
	& J_{ii}= -\sum_{e: e<f_i} \frac{\partial L_{(e,i)}}{\partial u_{i}}>0, \quad \forall f_i\in F ;\\
	&J_{ij}= J_{ji}=-\frac{\partial L_{(e,i)}}{\partial u_{j}}<0	, \quad \text{if there exists}\ e=f_i\cap f_j\in E;\\
	&J_{ij}= J_{ji}=0, \quad \text{otherwise} .
\end{align*}

Since
\begin{align*}
	J_{ii}-\sum_{j\ne i}|J_{ij}|
	&=-\sum_{e: e< f_i} \frac{\partial L_{(e,i)}}{\partial u_{i}}-\sum_{e: e< f_i}\frac{\partial L_{(e,i)}}{\partial u_{j}}\\
	&=-\sum_{e: e<f_i}(\frac{\partial L_{(e,i)}}{\partial u_{i}}+\frac{\partial L_{(e,i)}}{\partial u_{j}})\\
	&>0, \quad \forall f_i\in F ,
\end{align*}
then $J$ is a symmetric and strictly diagonally dominant matrix with positive diagonal entries. This implies that $J$ is positive definite and $\mathcal{E}(u)$ is strictly convex on $U$.
\end{proof}

\subsection{The proof of Theorem \ref{Thm of rigidity}.}

\begin{proof}
By Corollary $\ref{convex function 1}$,  the function $\mathcal{E}(u)=\sum_{e:e=f_i\cap f_j\in E}\mathcal{E}_{e}(u_{i}, u_{j})$ is strictly convex on the convex admissible space $U$ with the gradient $\bigtriangledown\mathcal{E}(u)=(-L_1,...,-L_{|F|})$, where $ u=(u_1,..., u_{|F|}).$ It is a classical result in analysis that the gradient map of a strictly convex $C^{2}$-function on a convex domain in $\mathbb{R}^{|F|}_{>0}$ is injective. Therefore, a spherical orthogonal ring pattern on the closed surface $S$ is uniquely determined by its modified combinatorial total geodesic curvature.
\end{proof}

\begin{remark}
Let $\Omega=\{(r_1,...,r_{|F|})\mid r_i\in(0,\frac{\pi}{2}),\forall \ i=1,\dots,|F|\}=(0,\frac{\pi}{2})^{|F|}$ and $P:\Omega\to \mathbb{R}^{|F|}_{>0}$, $(r_1,...,r_{|F|})\mapsto (L_1,...,L_{|F|})$. It is difficult to characterize the image of the function $P$ because we cannot obtain a satisfactory boundary for it. As a result, the existence for the spherical orthogonal ring pattern is not known.
\end{remark}

\section{Combinatorial curvature flows of spherical orthogonal ring patterns}\label{Combinatorial curvature flows of spherical orthogonal ring patterns}
\subsection{Combinatorial curvature flows}
The combinatorial Ricci flow was first introduced by Chow and Luo in \cite{C-L} for Thurston's Euclidean and hyperbolic circle patterns on closed surfaces.
The combinatorial Calabi flow originates from Ge's work \cite{G}, where he first formulated it for Thurston's Euclidean circle patterns on closed surfaces.
In this section, we introduce the combinatorial Ricci flow and combinatorial Calabi flow for spherical orthogonal ring patterns on closed surfaces and prove the local convergence of the solutions for these combinatorial curvature flows.

\begin{definition}
Given a function $\hat{L}=(\hat{L}_1,...,\hat{L}_{|F|})\in \mathbb{R}^{|F|}_{>0}$,
the combinatorial Ricci flow for the spherical orthogonal ring patterns is defined to be
\begin{equation}\label{Eq: CRF}
	\frac{\mathrm{d}r_i}{\mathrm{d}t}
	=(L_i-\hat{L}_i)\sin 2R_i, \quad \forall i\in \{1,...,|F|\},
\end{equation}
and the combinatorial Calabi flow for the spherical orthogonal ring patterns is defined to be
\begin{equation}\label{Eq: CCF}
	\frac{\mathrm{d}r_{i}}{\mathrm{d}t}
	=-(\Delta(L-\hat{L}))_i\sin 2R_i, \quad \forall i\in \{1,...,|F|\},
\end{equation}
where $\Delta=(\frac{\partial L_i}{\partial u_j})_{|F|\times |F|}$ is the discrete Laplace operator.
\end{definition}

Through the variable substitution (\ref{u_i}), the combinatorial Ricci flow (\ref{Eq: CRF}) can be rewritten as
\begin{equation}\label{CRF;u}
	\frac{\mathrm{d}u_i}{\mathrm{d}t}
	=L_i-\hat{L}_i, \quad \forall i\in \{1,...,|F|\},
\end{equation}
and the combinatorial Calabi flow (\ref{Eq: CCF}) takes the equivalent form of
\begin{equation}\label{CCF;u}
	\frac{\mathrm{d}u_{i}}{\mathrm{d}t}
	=-(\Delta(L-\hat{L}))_i, \quad \forall i\in \{1,...,|F|\}.
\end{equation}

Set
\begin{equation*}
	\tilde{\mathcal{E}}(u)=\mathcal{E}(u)+\sum_{i=1}^{|F|}\hat{L}_iu_i
\end{equation*}
and
\begin{equation*}
	\mathcal{C}(u)=\frac{1}{2}\sum_{i=1}^{|F|}(L_i-\hat{L}_i)^2.
\end{equation*}
It is easy to verify that the flow (\ref{CRF;u}) is the negative gradient flow of $\tilde{\mathcal{E}}(u)$ and the flow (\ref{CCF;u}) is the negative gradient flow of $\mathcal{C}(u)$. Moreover, both $\tilde{\mathcal{E}}(u)$ and $\mathcal{C}(u)$ are decreasing along the flows (\ref{CRF;u}) and (\ref{CCF;u}).

\subsection{Local convergence of combinatorial curvature flows}
Since $L$ is a smooth function of $u$, it follows from the Picard-Lindel\"of theorem that the combinatorial curvature flows (\ref{CRF;u}) and (\ref{CCF;u}) admit a unique local solution for any initial value. We state the Lyapunov's theorem as presented by Pontryagin in \cite{Pontryagin} as follows.

\begin{lemma}[\cite{Pontryagin}, Chapter 5]\label{Ly thm}
Let $U$ be an open set in $\mathbb{R}^{|F|}$ and $f\in C^1(U,\mathbb{R}^{|F|})$. Consider an autonomous ordinary differential system
\begin{equation}\label{ods}
	\dot{x}(t)=f(x(t)), \ x(t)\in U.
\end{equation}
Assume $x_0$ is an equilibrium state of this system, i.e. $f(x_0)=0$. If all eigenvalues of the matrix $\frac{\partial f}{\partial x}(x_0)$ have negative real part, then $x_0$ is an asymptotically stable point. More precisely, there exists a neighborhood $\tilde{U}\subset U$ of $x_0$ such that, for any initial value $ x(0)\in \tilde{U}$, the solution of (\ref{ods}) exists for all time and converges exponentially fast to $x_0$.
\end{lemma}

We have the following theorem on the long-time existence and convergence of combinatorial curvature flows.

\begin{theorem}
Given a function $\hat{L}\in \mathbb{R}^{|F|}_{>0}$, if the solution $u(t)$ of the combinatorial Ricci flow (\ref{CRF;u}) or the combinatorial Calabi flow (\ref{CCF;u}) exists for all time and converges to $\hat{u}\in U$, then $L(\hat{u})=\hat{L}$. Moreover, if $\hat{L}$ is attainable (i.e., $L(\hat{u})=\hat{L}$ for some $\hat{u}\in U$), then there exists a neighborhood $\tilde{U}\subset U$ of $\hat{u}$ such that for any initial value $u(0)\in \tilde{U}$, the solutions of the flows (\ref{CRF;u}) and (\ref{CCF;u}) exist for all time and converges exponentially to $\hat{u}$.
\end{theorem}

\begin{proof}
Suppose $u(t)$ is a solution of the combinatorial Ricci flow (\ref{CRF;u}). If $u(t)$ converges to $\hat{u}\in U$, i.e., $\lim_{t\to +\infty}u(t)=\hat{u}$, then by the continuity of $L$ with respect to $u$, the limit $\lim_{t\to +\infty}L(u(t))=L(\hat{u})$ exists. Furthermore, by the mean value theorem, for each integer $n$ there exists $t_n\in(n,n+1)$ such that
\[ u_i(n+1)-u_i(n)=\frac{\mathrm{d}u_i}{\mathrm{d}t}(t_n)=L_i(u(t_n))-\hat{L}_i. \]
Taking the limit on both sides, we have $\lim_{n\to +\infty} L_i(u(t_n))-\hat{L}_i=0$, which implies $L(\hat{u})=\hat{L}$.
Similarly, suppose $u(t)$ is a solution of the combinatorial Calabi flow (\ref{CCF;u}). If $u(t)$ converges to $\hat{u}\in U$, then $\lim_{t\to +\infty}L(u(t))=L(\hat{u})$ exists. Moreover, by the mean value theorem, for each integer $n$ there exists  $t_n\in(n,n+1)$ such that
\[ u_i(n+1)-u_i(n)=\frac{\mathrm{d}u_i}{\mathrm{d}t}(t_n)=-(\Delta(L(u(t_n))-\hat{L}))_i. \]
Taking the limit on both sides, we obtain $\lim_{n\to +\infty} \Delta(L(u(t_n))-\hat{L})=0$. Since the matrix $\Delta$ is negative definite, it follows that $L(\hat{u})=\hat{L}$.

Conversely, if $\hat{L}$ is attainable, then there exists $\hat{u}\in U$ such that $L(\hat{u})=\hat{L}$.
For the combinatorial Ricci flow (\ref{CRF;u}), let $f(u)=L-\hat{L}$. Then $\hat{u}$ is an equilibrium state of the flow (\ref{CRF;u}), and $Df(\hat{u})=\Delta(\hat{u})$ has all negative eigenvalues. Therefore, the conclusion follows from Lemma \ref{Ly thm}.
For the combinatorial Calabi flow (\ref{CCF;u}), let $g(u)=-\Delta(L-\hat{L})$. Then $\hat{u}$ is an equilibrium state of the flow (\ref{CCF;u}), and $Dg(\hat{u})=-\Delta^2(\hat{u})$ has all negative eigenvalues. Hence, by Lemma \ref{Ly thm}, the result holds.
\end{proof}

\begin{remark}
For the Euclidean and hyperbolic orthogonal ring patterns introduced by Bobenko-Hoffmann-R\"orig \cite{B-H-R} and Bobenko \cite{B} respectively, the combinatorial curvature flows analogous to the combinatorial curvature flows (\ref{CRF;u}) and (\ref{CCF;u}) can be defined similarly, and the local convergence of these flows can also be established.
\end{remark}

\end{document}